\documentclass[12pt,reqno]{amsart}
\usepackage{amsmath}
\usepackage{mathrsfs,amssymb,graphicx,verbatim,hyperref,mathtools}
\usepackage{paralist}

\usepackage{ifthen}

\newcommand{\ifsodaelse}[2]{\ifthenelse{\isundefined{\SODAF}}{#2}{#1}}
\ifsodaelse    {\usepackage{ltexpprt}\usepackage{amsmath}}{}
\usepackage{mathrsfs,amssymb,graphicx,verbatim}
\usepackage{paralist}

\renewcommand{\hat}{\widehat}
\newcommand\remove[1]{}
\newcommand{\Rad}{{\mathrm{\bf Rad}}}

\DeclareRobustCommand\widecheck[1]{{\mathpalette\@widecheck{#1}}}
\def\@widecheck#1#2{%
    \setbox\z@\hbox{\m@th$#1#2$}%
    \setbox\tw@\hbox{\m@th$#1%
       \widehat{%
          \vrule\@width\z@\@height\ht\z@
          \vrule\@height\z@\@width\wd\z@}$}%
    \dp\tw@-\ht\z@
    \@tempdima\ht\z@ \advance\@tempdima2\ht\tw@ \divide\@tempdima\thr@@
    \setbox\tw@\hbox{%
       \raise\@tempdima\hbox{\scalebox{1}[-1]{\lower\@tempdima\box
\tw@}}}%
    {\ooalign{\box\tw@ \cr \box\z@}}}

\newcommand{\A}{{\mathcal A}}
\newcommand{\tp}{\widehat{\otimes}}
\newcommand{\sign}{\mathrm{sign}}
\renewcommand{\d}{\delta}

\newcommand{\e}{\varepsilon}
\newcommand{\R}{\mathbb{R}}

\newcommand{\E}{\mathbb{E}}

\newcommand{\N}{\mathbb{N}}

\newcommand{\Prob}{\mathrm{Prob}}

\newtheorem{theorem}{Theorem}[section]
\newtheorem{lemma}[theorem]{Lemma}

\newtheorem{corollary}[theorem]{Corollary}

\newtheorem{definition}[theorem]{Definition}

\newcommand{\eqdef}{\stackrel{\mathrm{def}}{=}}
\date{}

\renewcommand{\le}{\leqslant}
\renewcommand{\ge}{\geqslant}

\newcommand\eps{\varepsilon}

\renewcommand{\epsilon}{\varepsilon}

\theoremstyle{remark}

\newcommand{\F}{\mathcal{F}}

\renewcommand{\S}{\mathbb{S}}

\title[LDC and cotype of tensor products]{Locally decodable codes and the failure of cotype for projective tensor products}
\thanks{J. B. was supported by an NWO Vici grant and the EU project QCS. A. N. was supported by NSF grant CCF-0832795, BSF grant
 2010021, the Packard Foundation and the Simons Foundation. Part of this work was completed while A. N. was visiting Universit\'e Pierre et Marie Curie, Paris, France. O. R. was supported by a European Research Council (ERC) Starting Grant. }

\author{Jop Bri\"et}
\address{Centrum Wiskunde \& Informatica (CWI), Science Park 123, 1098 SJ Amsterdam, The Netherlands}
\email{j.briet@cwi.nl}

\author{Assaf Naor}
\address{Courant Institute, New York University, 251 Mercer Street, New York NY 10012, USA}
\email{naor@cims.nyu.edu}

\author{Oded Regev}
\address{{\'E}cole normale sup{\'e}rieure,
D{\'e}partement d'informatique,
45 rue d'Ulm, Paris,
France}
\email{regev@di.ens.fr}

\date{}

\renewcommand{\S}{\mathcal{S}}
\renewcommand{\O}{\mathcal{O}}
\newcommand{\n}{{\|\cdot\|}}

\begin{document}

\begin{abstract}
It is shown that for every $p\in (1,\infty)$ there exists a Banach space $X$ of finite cotype such that the projective tensor product $\ell_p\tp X$ fails to have finite cotype. More generally, if $p_1,p_2,p_3\in (1,\infty)$ satisfy $\frac{1}{p_1}+\frac{1}{p_2}+\frac{1}{p_3}\le 1$ then $\ell_{p_1}\tp\ell_{p_2}\tp\ell_{p_3}$ does not have finite cotype. This is a proved via a connection to the theory of locally decodable codes.
\end{abstract}

\maketitle

\section{introduction}\label{sec:intro}

Throughout this paper all Banach spaces are assumed to be over the reals, though our results apply (with the same proofs) to complex Banach spaces as well. We shall use standard Banach space notation and terminology,  e.g., as in~\cite{AK05}. We shall also use the asymptotic notation $\lesssim,\gtrsim$
to indicate the corresponding inequalities up to universal constant
factors, and we shall denote equivalence up to universal constant
factors by $\asymp$, i.e., $A\asymp B$ is the same as $(A\lesssim
B)\wedge (A\gtrsim B)$.

The {\em projective tensor product} of two Banach spaces $(X,\|\cdot\|_X)$ and $(Y,\|\cdot\|_Y)$, denoted $X\tp Y$, is the completion of their algebraic tensor product $X\otimes Y$, equipped with the norm
\begin{multline*}
\|z\|_{X\tp Y}=\inf\left\{\sum_{i=1}^n \|x_i\|_X\cdot\|y_i\|_Y:\ \exists\, n\in \N,\exists \{(x_i,y_i)\}_{i=1}^n\subseteq X\times Y,\right.\\\left.\mathrm{such\ that\ }z=\sum_{i=1}^n x_i\otimes y_i\right\}.
\end{multline*}
Thus, if $X, Y$ are finite dimensional then the unit ball of $X\tp Y$ is the convex hull in $X\otimes Y$ of all the vectors of the form $x\otimes y$, where $x$ is a unit vector in $X$ and $y$ is a unit vector in $Y$.  To state two concrete examples of this construction, one always has $\ell_1\tp X=\ell_1(X)$, and $\ell_2\tp \ell_2$ can be naturally identified with the Schatten trace class $S_1$, i.e., the space of all compact operators $T:\ell_2\to \ell_2$, equipped with the norm $\|T\|_{S_1}=\mathrm{trace}( \sqrt{T^*T})$. For these facts, and much more information on projective tensor products, we refer to~\cite{Gro53,Rya02,DFS08}.

The literature contains a significant amount of work on the permanence of various key Banach space properties under projective tensor products; see the survey~\cite{DFS03} for some of the known results along these lines. Here we will be mainly concerned with geometric properties of projective tensor products with $L_p(\mu)$ spaces, $p\in (1,\infty)$, in which case examples of known results include that $L_p\tp X$ is weakly sequentially complete iff $X$ is~\cite{Lew77},  $L_p\tp X$ has the Radon-Nikod\'ym property iff $X$ does~\cite{BD01,Bu02}, and $L_p\tp X$ contains a copy of $c_0$ iff $X$ does~\cite{BD02}.

When one does not consider projective tensor products with $L_p(\mu)$ spaces, the above permanence properties are known to fail~\cite{Pis83,BP83}. Specifically, Bourgain and Pisier~\cite{BP83} showed that there exist a weakly sequentially complete Banach space $X$ with the Radon-Nikod\'ym property, such that $X\tp X$ contains a copy of $c_0$ (thus $X\tp X$ fails weak sequential completeness and the Radon-Nikod\'ym property).

Here we will be concerned with the permanence of finite cotype under projective tensor products. For $q\in [2,\infty)$, a Banach space $(X,\|\cdot\|_X)$ is said to have cotype $q$ if there exists $C\in (0,\infty)$ such that for every $n\in \N$ and every $x_1,\ldots,x_n\in X$ we have
\begin{equation}\label{eq:def cotype}
\left(\sum_{i=1}^n \|x_i\|_X^q\right)^{1/q}\le C\left(\E\left[\left\|\sum_{i=1}^n\e_i x_i\right\|_X^2\right]\right)^{1/2},
\end{equation}
where the expectation in~\eqref{eq:def cotype} is taken with respect to uniformly distributed $\e=(\e_1,\ldots,\e_n)\in \{-1,1\}^n$. Thus $L_p$ has cotype $\max\{p,2\}$ for for every $p\in [1,\infty)$ (see, e.g.,~\cite{AK05}). The infimum over those $C\in (0,\infty)$ for which~\eqref{eq:def cotype} holds true is denoted $C_q(X,\|\cdot\|_X)$, or, if the norm is clear from the context, simply $C_q(X)$. Given $k\in \N$ and a norm $\n$ on $\R^k$, it will also be convenient to denote $C_\n^{(q)}=C_q(\R^k,\n)$. If $X$ has cotype $q$ for some $q\in [2,\infty)$ then we say that $X$ has finite cotype, or simply that $X$ has cotype. The Maurey-Pisier theorem~\cite{MP76} implies that $X$ fails to have finite cotype if and only if it is universal in the sense that there exists $K\in (0,\infty)$ such that {\em all} finite dimensional Banach spaces embed into $X$ with distortion at most $K$ (equivalently, $\ell_\infty^n$ embeds into $X$ with distortion at most $K$ for all $n\in \N$). We are therefore interested in the permanence under projective tensor products of the failure of the above universality property.

Tomczak-Jaegermann proved~\cite{Tom74} that $\ell_2\tp\ell_2=S_1$ has cotype $2$, and Pisier proved~\cite{Pis92,Pis-92-clapem}  that if $p,q\in [2,\infty)$ then $L_p\tp L_q$ has cotype $\max\{p,q\}$ (see~\cite{Pis92} for a more general result along these lines). Other than these facts and the easy fact that $L_1\tp X=L_1(X)$ always inherits the cotype of $X$, we do not know of other permanence results for cotype under projective tensor products. In particular, it is open whether $L_p\tp L_q$ has finite cotype when $p\in (1,2)$ and $q\in (1,2]$, and whether $\ell_2\tp S_1=\ell_2\tp\ell_2\tp\ell_2$ has finite cotype (these questions are stated in~\cite{Pis-92-clapem}).

A remarkable theorem of Pisier~\cite{Pis83} asserts that there exist two Banach spaces $X$ and $Y$ of finite cotype such that $X\tp Y$ does not have finite cotype. Specifically, by a famous theorem of Bourgain~\cite{Bou84},  $L_1/H^1$ has cotype $2$ ($H^1$ is the closed span of $\{\theta \mapsto e^{2\pi in\theta}\}_{n=0}^\infty\subseteq L_1$), and Pisier constructs~\cite{Pis83} a Banach space $Z$ of cotype $2$ such that $Z\tp (L_1/H^1)$ contains a copy of $c_0$.

We have seen that projective tensor products with $L_p(\mu)$ spaces preserve a variety of geometric properties, but that similar results often fail for projective tensor products between general Banach spaces. In this vein, for $p\in (1,\infty)$ it was unknown whether $L_p\tp X$ has finite cotype if $X$ has finite cotype. This question was explicitly asked in~\cite[p.~59]{DFS03}, and here we answer it negatively by showing  that for every $p\in (1,\infty)$ there exists a Banach space $X$ of finite cotype such that $\ell_p\tp X$ fails to have finite cotype. Thus, $L_p\tp X$ contains copies $\{\ell_\infty^n\}_{n=1}^\infty$ with distortion bounded by a constant independent of $n$; contrast this statement with the theorem of Bu and Dowling~\cite{BD02} quoted above that asserts that $L_p\tp Y$ contains a copy of $c_0$ iff $Y$ itself contains a copy of $c_0$. Note that when $p=2$ we see that even the projective tensor product with Hilbert space need not preserve finite cotype.

Our main result is the following theorem.
\begin{theorem}\label{thm:lps no cotype}
Fix $p_1,p_2,p_3\in (1,\infty)$ such that
\begin{equation}\label{eq:pqr assumption}
\frac{1}{p_1}+\frac{1}{p_2}+\frac{1}{p_3}\le 1.
\end{equation}
Then $\ell_{p_1}\tp\ell_{p_2}\tp\ell_{p_3}$ does not have finite cotype. Moreover, there exists a universal constant $c\in (0,\infty)$ such that for every $p_1,p_2,p_3\in (1,\infty)$ satisfying~\eqref{eq:pqr assumption}, every $q\in [2,\infty)$, and every integer $n> 15$ we have
\begin{equation}\label{eq:cotype costant lower three tensor}
C_q\left(\ell_{p_1}^{n}\tp\ell_{p_2}^{n}\tp\ell_{p_3}^{n}\right)\gtrsim \frac{1}{\log n}\cdot \exp\left(\frac{c}{q}\cdot\frac{(\log\log n)^2}{\log\log\log n}\right).
\end{equation}
\end{theorem}

It follows from Theorem~\ref{thm:lps no cotype} that for every $p\in (1,\infty)$, if we set $X=\ell_{2p/(p-1)}\tp \ell_{2p/(p-1)}$ then $\ell_p\tp X$ fails to have finite cotype. By the result of Pisier~\cite{Pis92} quoted above, $X$ has finite cotype. Another notable consequence of Theorem~\ref{thm:lps no cotype} is that there exists a Banach space $Y$ such that $Y\tp Y$ has finite cotype yet $Y\tp Y\tp Y$ fails to have finite cotype. We conjecture that~\eqref{eq:cotype costant lower three tensor} is not sharp, leaving open the determination of the asymptotic behavior of, say, $C_q\left(\ell_{3}^{n}\tp\ell_{3}^{n}\tp\ell_{3}^{n}\right)$.

Our proof of Theorem~\ref{thm:lps no cotype} is based on a connection between cotype of tensor products and results from theoretical computer science, namely the theory of locally decodable codes. This link allows us to use (as a ``black box") delicate constructions that are available in the computer science literature in order to prove Theorem~\ref{thm:lps no cotype}. Our initial hope was to use this connection in the reverse direction, namely, to use Banach space theory to address an important question about the length of locally decodable codes, but it turned out that instead locally decodable codes can be used to address the question in Banach space theory described above. Nevertheless, there is hope that the connection presented below, when combined with geometric insights about tensor norms, might lead to improved lower bounds for locally decodable codes. This hope will be made explicit in the following section.

\subsection{Locally decodable codes and cotype}\label{sec:LDC}
Definition~\ref{def:LDC} below is due to Katz and Trevisan~\cite{KT00}; see the surveys~\cite{Tre04, Yek11} for more information on this notion (and the closely related notion of {\em private information retrieval}), as well as a description of some of its many applications in cryptography and computational complexity theory. We note that in the present paper no reference to Definition~\ref{def:LDC} will be made other than through the conclusion of Lemma~\ref{lem:no alg def} below.

\begin{definition}[$3$-query locally decodable code]\label{def:LDC}
Fix $m,n\in \N$ and $\phi,\theta\in (0,1/2)$. A function $$C:\{-1,1\}^m\to \{-1,1\}^{n}$$ is called a  $3$-query locally decodable code of quality $(\phi,\theta)$ if for every $t\in \{1,\ldots,m\}$ there exists a distribution $\A_t$ over $4$-tuples $(i,j,k,g)$, where $i,j,k \in \{1,\ldots,n\}$ and $g:\{-1,1\}^3 \to \{-1,1\}$, with the property that for every $\e\in \{-1,1\}^m$ and every $\d\in \{-1,1\}^n$ that differs from $C(\e)$ in at most $\phi\cdot n$ coordinates, with probability (with respect to the distribution $\A_t$) at least $\frac12+\theta$  we have $g(\d_i,\d_j,\d_k)=\e_t$.
\end{definition}

The motivation behind Definition~\ref{def:LDC} is as follows. Just like standard error correcting codes, locally decodable codes provide a way to encode an $m$-bit message into a longer $n$-bit codeword in a way that allows one to recover the original message from the codeword, even if it is corrupted in any  set of coordinates that isn't too large. However, while standard error correcting codes typically require reading essentially all the $n$ bits of the corrupted codeword in order to recover even one bit of the message, a $3$-query locally decodable code allows one to do this while reading only $3$ bits.

It is an important open question  to determine the asymptotic behavior in $m$ of the smallest $n$ for which $3$-query locally decodable codes exist (for some fixed $\phi,\theta$, say, $\phi=\theta=1/16$).
The best known upper bound, due to Efremenko~\cite{Efr09} using in part key ideas of Yekhanin~\cite{Yek08} and a combinatorial construction of Grolmusz~\cite{Gro00}, is that for every $m\in \N$ there exists an integer $n\in \N$ satisfying
\begin{equation}\label{eq:efremenko}
\log\log n\asymp \sqrt{\log m\log\log m},
\end{equation}
for which there exists  a code $C:\{-1,1\}^m\to \{-1,1\}^n$ which is $3$-query locally decodable  of quality $\left(\phi,\frac12-6\phi\right)$ for all $\phi\in (0,1/12)$. See~\cite{BeimelIKO12} for an improvement of the implicit constant factor in~\eqref{eq:efremenko}.

The best known lower bound, due to Woodruff~\cite{Woo07} as a logarithmic improvement over a lower bound of Kerenidis and de Wolf~\cite{KdW04}, is that for, say, $\phi=\theta=1/16$ we necessarily have
\begin{equation}\label{eq:woodruff}
n\gtrsim \frac{m^2}{\log m}.
\end{equation}

In what follows, given $n\in \N$ we let $e_1,\ldots,e_n$ be the standard coordinate basis of $\R^n$.  Let $\|\cdot\|$ be a norm on $\R^{n}\otimes \R^{n}\otimes \R^{n}$. For $K\in [1,\infty)$ say that $\|\cdot\|$ is {\em $K$-tensor-symmetric} if for every choice of permutations $\pi,\sigma,\tau\in S_{n}$, every choice of sign vectors $\e,\d,\eta\in \{-1,1\}^{n}$, and every choice of scalars  $\{a_{i,j,k}\}_{i=1}^{n}\subseteq \R$, we have
\begin{multline}\label{eq: def tensor symmetric}
\left\|\sum_{i=1}^{n}\sum_{j=1}^{n}\sum_{k=1}^{n}\e_i\d_j\eta_k a_{\pi(i),\sigma(j),\tau(k)} e_i\otimes e_j\otimes e_k\right\|\\\le K \left\|\sum_{i=1}^{n}\sum_{j=1}^{n}\sum_{k=1}^{n}a_{i,j,k} e_i\otimes e_j\otimes e_k\right\|.
\end{multline}

\begin{theorem}\label{thm:tensor symmetric}
Fix $m,n\in \N$ and $\phi,\theta\in (0,1/2)$. Suppose that there exists a $3$-query locally decodable code $C:\{-1,1\}^m\to \{-1,1\}^n$ of quality $(\phi,\theta)$. For every $K\in [1,\infty)$, if $\n$ is a $K$-tensor-symmetric norm on $\R^{3n}\otimes \R^{3n}\otimes \R^{3n}$ then for every $q\in [2,\infty)$ we have
$$
K^2C_\n^{(q)}\cdot\frac{\left\|\sum_{i=1}^{3n}\sum_{j=1}^{3n}\sum_{k=1}^{3n} e_i\otimes e_j\otimes e_k\right\|}{\left\|\sum_{i=1}^{3n}e_i\otimes e_i\otimes e_i\right\|}\gtrsim \frac{\phi\theta^2 m^{1/q}}{\log(n+1)}.
$$
\end{theorem}

Theorem~\ref{thm:tensor symmetric} implies that if, say, $\phi=\theta=1/16$ then there exists a universal constant $c\in (0,\infty)$ such that
\begin{equation}\label{eq:sup lower}
n\ge \sup_\n \sup_{q\in [2,\infty)} \exp\left(\frac{cm^{1/q}\left\|\sum_{i=1}^{3n}e_i\otimes e_i\otimes e_i\right\|}{C_\n^{(q)} \left\|\sum_{i=1}^{3n}\sum_{j=1}^{3n}\sum_{k=1}^{3n} e_i\otimes e_j\otimes e_k\right\|}\right),
\end{equation}
where the first supremum in~\eqref{eq:sup lower} is taken over all the $1$-tensor-symmetric norms $\|\cdot\|$ on  $\R^{3n}\otimes \R^{3n}\otimes \R^{3n}$. While our initial hope was to use~\eqref{eq:sup lower} to narrow the large gap between~\eqref{eq:efremenko} and~\eqref{eq:woodruff}, we do not know if there exists a norm on $\R^{3n}\otimes \R^{3n}\otimes \R^{3n}$ with respect to which~\eqref{eq:sup lower} exhibits an asymptotic improvement over~\eqref{eq:woodruff}. This question is arguably the most important question that the present paper leaves open. However, Theorem~\ref{thm:tensor symmetric} contains new information when one contrasts it with Efremenko's upper bound~\eqref{eq:efremenko}, thus yielding the following corollary.

\begin{corollary}\label{coro:cotype const lower}
There exists a universal constant $c\in (0,\infty)$  such that for every integer $n>15$, every $q\in [2,\infty)$ and every $K\in [1,\infty)$, any $K$-tensor-symmetric norm $\|\cdot\|$ on $\R^{n}\otimes \R^{n}\otimes \R^{n}$ satisfies
\begin{equation*}
C_\n^{(q)}\gtrsim \frac{\left\|\sum_{i=1}^{n}e_i\otimes e_i\otimes e_i\right\|}{\left\|\sum_{i=1}^{n}\sum_{j=1}^{n}\sum_{k=1}^{n} e_i\otimes e_j\otimes e_k\right\|}
\cdot\frac{\exp\left(\frac{c}{q}\cdot\frac{(\log\log n)^2}{\log\log\log n}\right)}{K^2\log n}.
\end{equation*}
\end{corollary}

Suppose that $p_1,p_2,p_3\in (1,\infty)$ satisfy~\eqref{eq:pqr assumption} and define $r\in [1,\infty)$ by
\begin{equation*}\label{eq:def s exponent}
\frac{1}{r}\eqdef \frac{1}{p_1}+\frac{1}{p_2}+\frac{1}{p_3}.
\end{equation*}
Denoting $e=\sum_{i=1}^ne_i\in \R^n$, we have
\begin{multline}\label{eq:three e}
\left\|\sum_{i=1}^{n}\sum_{j=1}^{n}\sum_{k=1}^{n}e_i\otimes e_j\otimes e_k\right\|_{\ell_{p_1}^{n}\tp\ell_{p_2}^{n}\tp\ell_{p_3}^{n}}=\left\|e\otimes e\otimes e\right\|_{\ell_{p_1}^{n}\tp\ell_{p_2}^{n}\tp\ell_{p_3}^{n}}\\=\|e\|_{\ell_{p_1}^n}\cdot\|e\|_{\ell_{p_2}^n}\cdot\|e\|_{\ell_{p_3}^n}=n^{1/p_1}\cdot n^{1/p_2}\cdot n^{1/p_3}=n^{1/r}.
\end{multline}
It is also well known (see, e.g., Theorem 1.3 in~\cite{AF96}) that
\begin{equation}\label{eq:our ratio is 1}
\left\|\sum_{i=1}^{n}e_i\otimes e_i\otimes e_i\right\|_{\ell_{p_1}^{n}\tp\ell_{p_2}^{n}\tp\ell_{p_3}^{n}}= n^{1/r}.
\end{equation}
Since $\ell_{p_1}^{n}\tp\ell_{p_2}^{n}\tp\ell_{p_3}^{n}$ is $1$-tensor-symmetric,  it follows from~\eqref{eq:three e} and~\eqref{eq:our ratio is 1} that Theorem~\ref{thm:lps no cotype} is a consequence of Corollary~\ref{coro:cotype const lower}.

\section{Preliminaries}\label{sec:prem}
In this section we briefly recall some standard notation and results on vector-valued Fourier analysis.

For $n\in \N$, the Walsh functions $\{W_A:\{-1,1\}^n\to \{-1,1\}\}_{A\subseteq \{1,\ldots,n\}}$ are given by $W_A(\e)=\prod_{i\in A}\e_i$. If $(X,\|\cdot\|_X)$ is a Banach space then for every $f:\{-1,1\}^n\to X$ and $A\subseteq \{1,\ldots,n\}$ we write
\begin{equation}\label{eq:fourier coeff def}
\widehat{f}(A)=\E\left[W_A(\e)f(\e)\right],
\end{equation}
where the expectation in~\eqref{eq:fourier coeff def} is with respect to $\e\in \{-1,1\}^n$ chosen uniformly at random. Then,
$$
\forall\, \e\in \{-1,1\}^n,\quad f(\e)=\sum_{A\subseteq\{1,\ldots,n\}} W_A(\e) \widehat{f}(A).
$$

The {\em Rademacher projection} of $f$, denoted $\Rad(f):\{-1,1\}^n\to X$, is defined by
$$
\forall\, \e\in \{-1,1\}^n,\quad\Rad(f)(\e)=\sum_{i=1}^n \e_i\hat{f}(\{i\}).
$$
Pisier's  famous bound on the $K$-convexity constant~\cite{Pis80} asserts that if $X$ is finite dimensional then every $f:\{-1,1\}^n\to X$ satisfies
\begin{equation}\label{eq:pisier K convex bound quote}
\sqrt{\E\left[\left\|\Rad(f)(\e)\right\|_X^2\right]}\lesssim \log \left(\dim(X)+1\right) \cdot \sqrt{\E\left[\left\|f(\e)\right\|_X^2\right]}.
\end{equation}
Recall that the implied constant in~\eqref{eq:pisier K convex bound quote} is universal, and thus it does not depend on $n$, $f$, $(X,\|\cdot\|_X)$ or $\dim(X)$. Bourgain proved~\cite{Bou84-Kconvex} that the logarithmic dependence on $\dim(X)$ in~\eqref{eq:pisier K convex bound quote} cannot be improved in general.

\section{Relating locally decodable codes to cotype}\label{sec:LDC-cotype}

Locally decodable codes will be used in what follows via the following lemma which is a slight variant of a result that appears in Appendix B of~\cite{BRW08} (the proof in~\cite{BRW08} is itself  a variant of an argument in~\cite{KdW04}).
\begin{lemma}[\cite{BRW08}]\label{lem:no alg def}
Fix $m,n\in \N$ and $\phi,\theta\in (0,1/2)$. Suppose that $C:\{-1,1\}^m\to \{-1,1\}^n$ is a $3$-query locally decodable code  of quality $(\phi,\theta)$. Then there exists a function $C':\{-1,1\}^m\to \{-1,1\}^{3n}$ with the following properties. For every $i\in \{1,\ldots,m\}$ there exist three  permutations $\pi_i,\sigma_i,\tau_i\in S_{3n}$,
and for every  $j\in \{1,\ldots,\lceil \phi\theta n/9\rceil\}$ there exists a sign
$
\d_i^j\in \{-1,1\},
$
such that if $\e\in \{-1,1\}^m$ is chosen uniformly at random then with probability at least $\frac12 +\frac{\theta}{16}$ we have
\begin{equation}\label{eq:LDC def}
 C'(\e)_{\pi_i\left(j\right)}C'(\e)_{\sigma_i\left(j\right)}C'(\e)_{\tau_i\left(j\right)}=\d_i^j\e_i.
\end{equation}
\end{lemma}
\begin{proof}
By Appendix B of~\cite{BRW08}, for every $i\in \{1,\ldots,m\}$ there exists a family of nonempty disjoint subsets $\F_i$ of $\{1,\ldots, n\}$ such that
\begin{itemize}
\item $|\F_i|\ge\frac{ \phi\theta n}{9}$,
\item each $S\in \F_i$ satisfies $|S|\le 3$,
\item for each $S\in \F_i$ there exists a sign $\d_i(S)\in \{-1,1\}$ such that if $\e\in \{-1,1\}^m$ is chosen uniformly at random then
\begin{equation}\label{eq:expectation product}
\Prob\left[\prod_{s\in S} C(\e)_s=\d_i(S)\e_i\right]\ge \frac12 +\frac{\theta}{16}.
\end{equation}
\end{itemize}

 Define $C':\{-1,1\}^m\to \{-1,1\}^{3n}$ by setting the first $n$ coordinates of $C'(x)$ to be equal to $C(x)$, and defining the remaining $2n$ coordinates of $C'(x)$ to be equal to $1$. One can then add $3-|S|$ elements from  $\{n+1,\ldots,3n\}$ to each set $S\in \F_i$ so that that $\F_i$ becomes a family of disjoint subsets of $\{1,\ldots, 3n\}$ of size equal to $3$, while not changing the validity of~\eqref{eq:expectation product} with $C$ replaced by $C'$. Now, there are  $\pi_i, \sigma_i,\tau_i\in S_{3n}$ such that $\F_i= \{\{\pi_i(j),\sigma_i(j),\tau_i(j)\}\}_{j=1}^{|\F_i|}$. Writing  $\d_i^j=\d_i(\{\pi_i(j),\sigma_i(j),\tau_i(j)\})$, the validity of~\eqref{eq:LDC def} with probability at least $\frac12 +\frac{\theta}{16}$ is the same as~\eqref{eq:expectation product}.
\end{proof}

Fix $n\in \N$ and let $\|\cdot\|$ be a seminorm on $\left(\R^n\right)^{\otimes 3}\eqdef \R^{n}\otimes \R^{n}\otimes \R^{n}$. Write
\begin{equation}\label{eq:defO}
\O_\n\eqdef \max_{\e\in \{-1,1\}^{n}}\|\e\otimes \e\otimes \e\|.
\end{equation}
For $\alpha,\beta\in (0,1)$  consider the subset $S(\alpha,\beta)\subseteq \left(\R^n\right)^{\otimes 3}$ defined by
\begin{multline}\label{eq:B set}
S(\alpha,\beta)\\\eqdef \bigcup_{\pi,\sigma,\tau\in S_n}\bigcap_{j=1}^{\left\lceil \beta n\right\rceil} \left\{x\in \left(\R^n\right)^{\otimes 3}:\ \left|\left\langle e_{\pi\left(j\right)}\otimes e_{\sigma\left(j\right)}\otimes e_{\tau\left(j\right)},x\right\rangle\right|\ge \alpha\right\},
\end{multline}
and
write
\begin{equation}\label{eq:defS}
\S_\n(\alpha,\beta)\eqdef \min_{x\in S(\alpha,\beta)}\|x\|.
\end{equation}

\begin{theorem}\label{thm:main LDC cotype}
Fix $m,n\in \N$ and $\phi,\theta\in (0,1/2)$. Suppose that there exists a $3$-query locally decodable code $C:\{-1,1\}^m\to \{-1,1\}^n$ of quality $(\phi,\theta)$. Then for every seminorm $\n$ on $\R^{3n}\otimes \R^{3n}\otimes \R^{3n}$ and every $q\in [2,\infty)$ we have
\begin{equation}\label{eq:COS}
\frac{m^{1/q}}{\log (n+1)}\lesssim \frac{C_\n^{(q)}\cdot \O_\n}{\S_\n(\theta/8,\phi\theta /27)}.
\end{equation}
\end{theorem}

\begin{proof}  Let $C':\{-1,1\}^m\to \{-1,1\}^{3n}$ be the function from Lemma~\ref{lem:no alg def}. Define $f:\{-1,1\}^m\to \R^{3n}\otimes \R^{3n}\otimes \R^{3n}$ by $$f(\e)= C'(\e)\otimes C'(\e)\otimes C'(\e).$$ Recalling the definition~\eqref{eq:defO}, we have $\|f(\e)\|\le \O_\n$ for all $\e\in \{-1,1\}^m$. Combined with Pisier's bound on the $K$-convexity constant~\eqref{eq:pisier K convex bound quote}, we therefore have
\begin{multline}\label{eq:use cotype}
\log\left((3n)^3+1\right)\cdot \O_\n\gtrsim \left(\E\left\|\sum_{i=1}^{m} \e_i\hat{f}(\{i\})\right\|^2\right)^{1/2}\\\ge \frac{1}{C_\n^{(q)}}\left(\sum_{i=1}^m\left\|\hat{f}(\{i\})\right\|^q\right)^{1/q},
\end{multline}
where in the last step of~\eqref{eq:use cotype} we used the definition of the cotype $q$ constant $C_\n^{(q)}$.

Using the notation of Lemma~\ref{lem:no alg def}, fix $i\in \{1,\ldots,m\}$ and for every $j\in \{1,\ldots,\lceil \phi\theta n/9\rceil\}$ write
\begin{equation}\label{eq:def Pij}
P_i^j\eqdef \Prob\left[\left\langle  e_{\pi_i\left(j\right)}\otimes e_{\sigma_i\left(j\right)}\otimes e_{\tau_i\left(j\right)},f(\e)\right\rangle =\d_i^j\e_i\right],
\end{equation}
where the probability in~\eqref{eq:def Pij} is over $\e\in \{-1,1\}^m$ chosen uniformly at random. Recalling the definition of $f$, it follows from~\eqref{eq:LDC def} that
\begin{equation}\label{eq:def prob assumptions}
\forall (i,j)\in \{1,\ldots,m\}\times \{1,\ldots,\lceil \phi\theta n/9\rceil\},\quad P_i^j\ge \frac{1}{2}+\frac{\theta}{16}.
\end{equation}
Now, for every $(i,j)\in \{1,\ldots,m\}\times \{1,\ldots,\lceil \phi\theta n/9\rceil\}$ we have
\begin{eqnarray}\label{eq:redo expectation computation}
&&\!\!\!\!\!\!\!\!\!\!\!\!\!\!\!\!\!\!\!\!\!\!\!\!\!\!\!\!\!\!\!\!\!\!\nonumber\d_i^j \left\langle e_{\pi_i\left(j\right)}\otimes e_{\sigma_i\left(j\right)}\otimes e_{\tau_i\left(j\right)},\hat{f}(\{i\})\right\rangle
\\&=&\nonumber \left\langle \d_i^j e_{\pi_i\left(j\right)}\otimes e_{\sigma_i\left(j\right)}\otimes e_{\tau_i\left(j\right)},\E\left[\e_i f(\e)\right]\right\rangle
\\&=&\nonumber
\E\left[\d_i^j\e_i\left\langle  e_{\pi_i\left(j\right)}\otimes e_{\sigma_i\left(j\right)}\otimes e_{\tau_i\left(j\right)},f(\e)\right\rangle\right]\\&\stackrel{\eqref{eq:def Pij}}{=}&P_i^j-\left(1-P_i^j\right)\stackrel{\eqref{eq:def prob assumptions}}{\ge} \frac{\theta}{8}.
\end{eqnarray}
Recalling~\eqref{eq:B set}, it follows from~\eqref{eq:redo expectation computation} that $\hat{f}(\{i\})\in S(\theta/8,\phi\theta /27)$ for all $i\in \{1,\ldots,m\}$. The definition~\eqref{eq:defS} therefore implies that
$$
\min_{i\in \{1,\ldots, m\}}\left\|\hat{f}(\{i\})\right\|\ge \S_\n(\theta/8,\phi\theta /27),
$$
which gives the desired estimate~\eqref{eq:COS} due to~\eqref{eq:use cotype}.
\end{proof}

By substituting~\eqref{eq:efremenko} into Theorem~\ref{thm:main LDC cotype} we deduce that any seminorm on $\R^{n}\otimes \R^{n}\otimes \R^{n}$ must obey the following nontrivial restriction.

\begin{corollary}\label{cor:n_m} There exist universal constants $\alpha,\beta, c\in (0,1)$ such that for every integer  $n> 15$, if $\n$ is a seminorm on $\R^{n}\otimes \R^{n}\otimes \R^{n}$ and $q\in [2,\infty)$ then
\begin{equation}\label{eq:lower general norm}
\frac{C_\n^{(q)}\cdot \O_\n}{\S_\n(\alpha,\beta)}\gtrsim \frac{1}{\log n}\cdot\exp\left(\frac{c}{q}\cdot \frac{(\log \log n)^2}{\log\log\log n}\right),
\end{equation}
\end{corollary}

\begin{proof} By Efremenko's bound~\eqref{eq:efremenko} combined with Theorem~\ref{thm:main LDC cotype}, there exist universal constants $\alpha_1,\beta_1,c_1\in (0,1)$ and a sequence of integers $\{n_m\}_{m=3}^\infty\subseteq \N$ satisfying
\begin{equation}\label{eq:the sequence is dense}
\log\log n_m\asymp\sqrt{\log m \log\log m},
 \end{equation}
 such that for every integer  $m\ge 3$ and $q\in [2,\infty)$, if $|\cdot|$  is a seminorm on $\R^{3n_m}\otimes\R^{3n_m}\otimes \R^{3n_m}$ then
\begin{equation}\label{eq:triple norm}
 \frac{C_{|\cdot|}^{(q)}\cdot \O_{|\cdot|}}{ \S_{|\cdot|}(\alpha_1,\beta_1)}\gtrsim \frac{m^{1/q}}{\log n_m}\stackrel{\eqref{eq:the sequence is dense}}{\gtrsim} \frac{1}{\log n_m}\cdot\exp\left(\frac{c_1}{q}\cdot \frac{(\log \log n_m)^2}{\log\log\log n_m}\right).
\end{equation}

Due to~\eqref{eq:the sequence is dense}, there is $N_0\in \N$ such that if $n\ge N_0$ then there exists an integer $m\ge 3$ for which
\begin{equation}\label{eq:n in between}
 3\left(1-\frac{\beta_1}{7}\right)n_m\le n\le 3n_m.
 \end{equation}
 Observe that  by adjusting the constant $c$, the desired asymptotic inequality~\eqref{eq:lower general norm} holds true if $n\in (15,N_0)$. We may therefore assume that $n\ge N_0$, in which case we apply~\eqref{eq:triple norm} to the trivial extension of $\|\cdot\|$ to a seminorm $|\cdot|$ on  $\R^{3n_m}\otimes\R^{3n_m}\otimes \R^{3n_m}$, i.e., for every $\{a_{i,j,k}\}_{i,j,k=1}^{3n_m}$ set
$$
\left|\sum_{i=1}^{3n_m}\sum_{j=1}^{3n_m}\sum_{k=1}^{3n_m}a_{i,j,k} e_i\otimes e_j\otimes e_k\right|\eqdef \left\|\sum_{i=1}^{n}\sum_{j=1}^{n}\sum_{k=1}^{n}a_{i,j,k} e_i\otimes e_j\otimes e_k\right\|.
$$
 Then $C_{|\cdot|}^{(q)}=C_{\|\cdot\|}^{(q)}$ and $\O_{|\cdot|}=\O_{\|\cdot\|}$, and due to~\eqref{eq:n in between} we also have $\S_{|\cdot|}(\alpha_1,\beta_1)\ge \S_{\|\cdot\|}(\alpha_1,\beta_1/2)$. The desired estimate~\eqref{eq:lower general norm} is therefore a consequence of~\eqref{eq:triple norm}.
\end{proof}

Due to the following simple lemma, Theorem~\ref{thm:tensor symmetric} and Corollary~\ref{coro:cotype const lower}  follow from Theorem~\ref{thm:main LDC cotype} and  Corollary~\ref{cor:n_m}, respectively.

\begin{lemma}
Fix $n\in \N$ and $\alpha,\beta\in (0,1)$. For $K\in [1,\infty)$, if $\n$ is a $K$-tensor-symmetric norm on  $\R^{n}\otimes \R^{n}\otimes \R^{n}$ then
\begin{equation}\label{eq:O upper K}
\O_\n\le K\left\|\sum_{i=1}^n\sum_{j=1}^n\sum_{k=1}^n e_i\otimes e_j\otimes e_k\right\|.
\end{equation}
and
\begin{equation}\label{eq:S upper K}
\S_\n(\alpha,\beta)\ge \frac{\alpha\beta}{K}\left\|\sum_{i=1}^ne_i\otimes e_i\otimes e_i\right\|.
\end{equation}
\end{lemma}

\begin{proof}
\eqref{eq:O upper K} is an immediate consequence of the definitions~\eqref{eq: def tensor symmetric} and~\eqref{eq:defO}. Next, fix $x\in S(\alpha,\beta)$. Writing $x=\sum_{i=1}^n\sum_{j=1}^n\sum_{k=1}^n a_{i,j,k} e_i\otimes e_j \otimes e_k$ for some $\{a_{i,j,k}\}_{i,j,k=1}^n\subseteq \R$, and recalling~\eqref{eq:B set}, there exist $\pi,\sigma,\tau\in S_n$ such that
\begin{equation}\label{eq:bigger than alpha}
\forall\, i\in \left\{1,\ldots,\lceil \beta n\rceil\right\},\quad \left|a_{\pi(i),\sigma(i),\tau(i)}\right|\ge \alpha.
\end{equation}
For $\e,\d,\eta\in \{-1,1\}^n$ and $\rho\in S_n$  define
\begin{align}\label{eq:define new vectors correlated}
x_{\e,\d,\eta}^\rho\eqdef \sum_{i=1}^n\sum_{j=1}^n\sum_{k=1}^{n} \e_i\d_j s_k(\e,\d,\eta,\rho)  a_{\pi\circ\rho(i),\sigma\circ\rho(j),\tau\circ\rho(k)} e_i\otimes e_j\otimes e_k,
\end{align}
where
\begin{align*}\label{eq:definition of s}
s_k(\e,\d,\eta,\rho) \eqdef \begin{cases}
\eta_k & \mbox{if } \rho(k)>\lceil \beta n \rceil, \\
\eps_k \delta_k \sign(a_{\pi \circ \rho(k),\sigma\circ \rho(k),\tau\circ \rho(k)}) & \mbox{otherwise}.
\end{cases}
\end{align*}
Thus, if $(\e,\d,\eta,\rho)\in \{-1,1\}^n\times  \{-1,1\}^n\times  \{-1,1\}^n\times S_n$ is chosen uniformly at random then
\begin{equation}\label{eq:new vector identity}
\E\left[x_{\e,\d,\eta}^\rho\right]=\frac{\sum_{i=1}^{\lceil \beta n\rceil}\left|a_{\pi(i),\sigma(i),\tau(i)}\right|}{n}\sum_{j=1}^n e_j\otimes e_j\otimes e_j.
\end{equation}
Consequently,
\begin{multline}\label{eq:all x}
\alpha\beta \left\|\sum_{j=1}^n e_j\otimes e_j\otimes e_j\right\|\stackrel{\eqref{eq:bigger than alpha}\wedge\eqref{eq:new vector identity}}{\le} \left\|\E\left[x_{\e,\d,\eta}^\rho\right]\right\|\\\le \E\left[\left\|x_{\e,\d,\eta}^\rho\right\|\right]\stackrel{\eqref{eq: def tensor symmetric}\wedge \eqref{eq:define new vectors correlated}}{\le} K\|x\|.
\end{multline}
Recalling~\eqref{eq:defS}, the validity of~\eqref{eq:all x} for all $x\in S(\alpha,\beta)$ implies~\eqref{eq:S upper K}. \end{proof}

\subsection*{Acknowledgements} We are very grateful to Gilles Pisier for his encouragement and many insightful suggestions.

\bibliographystyle{abbrv}
\bibliography{tensor-cotype}

\begin{thebibliography}{10}

\bibitem{AK05}
F.~Albiac and N.~J. Kalton.
\newblock {\em Topics in {B}anach space theory}, volume 233 of {\em Graduate
  Texts in Mathematics}.
\newblock Springer, New York, 2006.

\bibitem{AF96}
A.~Arias and J.~D. Farmer.
\newblock On the structure of tensor products of {$l_p$}-spaces.
\newblock {\em Pacific J. Math.}, 175(1):13--37, 1996.

\bibitem{BeimelIKO12}
A.~Beimel, Y.~Ishai, E.~Kushilevitz, and I.~Orlov.
\newblock Share conversion and private information retrieval.
\newblock In {\em Proc. 27th IEEE Conf. on Computational Complexity (CCC'12)},
  pages 258--268, 2012.

\bibitem{BRW08}
A.~Ben-Aroya, O.~Regev, and R.~de~Wolf.
\newblock A hypercontractive inequality for matrix-valued functions with
  applications to quantum computing and {L}{D}{C}s.
\newblock In {\em 49th Annual IEEE Symposium on Foundations of Computer
  Science}, pages 477--486, 2008.
\newblock Available at \url{http://arxiv.org/abs/0705.3806}.

\bibitem{Bou84}
J.~Bourgain.
\newblock New {B}anach space properties of the disc algebra and {$H^{\infty
  }$}.
\newblock {\em Acta Math.}, 152(1-2):1--48, 1984.

\bibitem{Bou84-Kconvex}
J.~Bourgain.
\newblock On martingales transforms in finite-dimensional lattices with an
  appendix on the {$K$}-convexity constant.
\newblock {\em Math. Nachr.}, 119:41--53, 1984.

\bibitem{BP83}
J.~Bourgain and G.~Pisier.
\newblock A construction of {${\mathcal L}_{\infty }$}-spaces and related
  {B}anach spaces.
\newblock {\em Bol. Soc. Brasil. Mat.}, 14(2):109--123, 1983.

\bibitem{Bu02}
Q.~Bu.
\newblock Observations about the projective tensor product of {B}anach spaces.
  {II}. {$L^p(0,1)\widehat\otimes X,\ 1<p<\infty$}.
\newblock {\em Quaest. Math.}, 25(2):209--227, 2002.

\bibitem{BD01}
Q.~Bu and J.~Diestel.
\newblock Observations about the projective tensor product of {B}anach spaces.
  {I}. {$l^p\widehat{\otimes}X,\ 1<p<\infty$}.
\newblock {\em Quaest. Math.}, 24(4):519--533, 2001.

\bibitem{BD02}
Q.~Bu and P.~N. Dowling.
\newblock Observations about the projective tensor product of {B}anach spaces.
  {III}. {$L^p[0,1]\hat\otimes X,\ 1<p<\infty$}.
\newblock {\em Quaest. Math.}, 25(3):303--310, 2002.

\bibitem{DFS03}
J.~Diestel, J.~Fourie, and J.~Swart.
\newblock The projective tensor product. {I}.
\newblock In {\em Trends in {B}anach spaces and operator theory ({M}emphis,
  {TN}, 2001)}, volume 321 of {\em Contemp. Math.}, pages 37--65. Amer. Math.
  Soc., Providence, RI, 2003.

\bibitem{DFS08}
J.~Diestel, J.~H. Fourie, and J.~Swart.
\newblock {\em The metric theory of tensor products}.
\newblock American Mathematical Society, Providence, RI, 2008.
\newblock Grothendieck's r{\'e}sum{\'e} revisited.

\bibitem{Efr09}
K.~Efremenko.
\newblock 3-query locally decodable codes of subexponential length.
\newblock In {\em S{TOC}'09---{P}roceedings of the 2009 {ACM} {I}nternational
  {S}ymposium on {T}heory of {C}omputing}, pages 39--44. ACM, New York, 2009.

\bibitem{Gro00}
V.~Grolmusz.
\newblock Superpolynomial size set-systems with restricted intersections mod 6
  and explicit {R}amsey graphs.
\newblock {\em Combinatorica}, 20(1):71--85, 2000.

\bibitem{Gro53}
A.~Grothendieck.
\newblock R\'esum\'e de la th\'eorie m\'etrique des produits tensoriels
  topologiques.
\newblock {\em Bol. Soc. Mat. S\~ao Paulo}, 8:1--79, 1953.

\bibitem{KT00}
J.~Katz and L.~Trevisan.
\newblock On the efficiency of local decoding procedures for error-correcting
  codes.
\newblock In {\em Proceedings of the {T}hirty-{S}econd {A}nnual {ACM}
  {S}ymposium on {T}heory of {C}omputing}, pages 80--86 (electronic), New York,
  2000. ACM.

\bibitem{KdW04}
I.~Kerenidis and R.~de~Wolf.
\newblock Exponential lower bound for 2-query locally decodable codes via a
  quantum argument.
\newblock {\em J. Comput. System Sci.}, 69(3):395--420, 2004.

\bibitem{Lew77}
D.~R. Lewis.
\newblock Duals of tensor products.
\newblock In {\em Banach spaces of analytic functions ({P}roc. {P}elczynski
  {C}onf., {K}ent {S}tate {U}niv., {K}ent, {O}hio, 1976)}, pages 57--66.
  Lecture Notes in Math., Vol. 604. Springer, Berlin, 1977.

\bibitem{MP76}
B.~Maurey and G.~Pisier.
\newblock S\'eries de variables al\'eatoires vectorielles ind\'ependantes et
  propri\'et\'es g\'eom\'etriques des espaces de {B}anach.
\newblock {\em Studia Math.}, 58(1):45--90, 1976.

\bibitem{Pis80}
G.~Pisier.
\newblock Un th\'eor\`eme sur les op\'erateurs lin\'eaires entre espaces de
  {B}anach qui se factorisent par un espace de {H}ilbert.
\newblock {\em Ann. Sci. \'Ecole Norm. Sup. (4)}, 13(1):23--43, 1980.

\bibitem{Pis83}
G.~Pisier.
\newblock Counterexamples to a conjecture of {G}rothendieck.
\newblock {\em Acta Math.}, 151(3-4):181--208, 1983.

\bibitem{Pis92}
G.~Pisier.
\newblock Factorization of operator valued analytic functions.
\newblock {\em Adv. Math.}, 93(1):61--125, 1992.

\bibitem{Pis-92-clapem}
G.~Pisier.
\newblock Random series of trace class operators.
\newblock In {\em Proceedings {C}uarto {C}{L}{A}{P}{E}{M} {M}exico 1990.
  {C}ontribuciones en probabilidad y estadistica matematica}, pages 29--42,
  1992.
\newblock Available at \url{http://arxiv.org/abs/1103.2090}.

\bibitem{Rya02}
R.~A. Ryan.
\newblock {\em Introduction to tensor products of {B}anach spaces}.
\newblock Springer Monographs in Mathematics. Springer-Verlag London Ltd.,
  London, 2002.

\bibitem{Tom74}
N.~Tomczak-Jaegermann.
\newblock The moduli of smoothness and convexity and the {R}ademacher averages
  of trace classes {$S_{p}(1\leq p<\infty )$}.
\newblock {\em Studia Math.}, 50:163--182, 1974.

\bibitem{Tre04}
L.~Trevisan.
\newblock Some applications of coding theory in computational complexity.
\newblock In {\em Complexity of computations and proofs}, volume~13 of {\em
  Quad. Mat.}, pages 347--424. Dept. Math., Seconda Univ. Napoli, Caserta,
  2004.

\bibitem{Woo07}
D.~P. Woodruff.
\newblock New lower bounds for general locally decodable codes.
\newblock {\em Electronic Colloquium on Computational Complexity (ECCC)},
  14(006), 2007.

\bibitem{Yek08}
S.~Yekhanin.
\newblock Towards 3-query locally decodable codes of subexponential length.
\newblock {\em J. ACM}, 55(1):Art. 1, 16, 2008.

\bibitem{Yek11}
S.~Yekhanin.
\newblock Locally decodable codes.
\newblock {\em Found. Trends Theor. Comput. Sci.}, 7(1):1--117, 2011.

\end{thebibliography}
\end{document}